\documentclass[12pt]{amsart}
\usepackage[utf8]{inputenc}

\usepackage{amsmath,amsthm,amssymb,mathscinet,bbm}
\usepackage{parskip}
\usepackage{geometry}
\usepackage{xcolor} 

\usepackage{mathtools} 
\usepackage[hidelinks]{hyperref}
\author{Paul Pollack} 
\author{Akash Singha Roy}
\address{Department of Mathematics \\ University of Georgia \\ Athens, GA 30602}
\email{pollack@uga.edu}
\email{akash01s.roy@gmail.com}


\renewcommand\phi\varphi
\usepackage{accents}

\renewcommand{\pod}[1]{\allowbreak\mathchoice
  {\if@display \mkern 18mu\else \mkern 8mu\fi (#1)}
  {\if@display \mkern 18mu\else \mkern 8mu\fi (#1)}
  {\mkern4mu(#1)}
  {\mkern4mu(#1)}
}
\usepackage{graphicx}
\DeclareMathAlphabet{\curly}{U}{rsfs}{m}{n}
\newcommand{\con}{{\textrm{con}}}
\newcommand{\inc}{{ \textrm{inc}}}
\newcommand{\1}{\mathbbm{1}}

\newcommand{\F}{\mathbb{F}}
\newcommand\Z{\mathbb{Z}}

\newcommand\Vm{\V_{q, a, m}}
\newcommand\VmSt{\V_{q, a, m}^*}
\newcommand\Vw{\V_q(w)}
\newcommand\Vwst{\V_q^*(w)}

\renewcommand\v{\mathbf{v}}
\newcommand\V{\mathcal{V}}
\newcommand\Q{\mathbb{Q}}

\newtheorem{thm}{Theorem}[section]

\newtheorem{prop}[thm]{Proposition}
\newtheorem{lem}[thm]{Lemma}
\newtheorem*{hypa}{Hypothesis A}
\newtheorem*{hypb}{Hypothesis B}

\theoremstyle{remark}
\newtheorem*{rmk}{Remark}

\numberwithin{equation}{section}
\begin{document}
\title[Polynomially-defined multiplicative functions]{Distribution in coprime residue classes of polynomially-defined multiplicative functions}

\begin{abstract} An integer-valued multiplicative function $f$ is said to be \textsf{polynomially-defined} if there is a nonconstant separable polynomial $F(T)\in \Z[T]$ with $f(p)=F(p)$ for all primes $p$. We study the distribution in coprime residue classes of polynomially-defined multiplicative functions, establishing equidistribution results allowing a wide range of uniformity in the modulus $q$. 
For example, we show that the values $\phi(n)$, sampled over integers $n \le x$ with $\phi(n)$ coprime to $q$, are asymptotically equidistributed among the coprime classes modulo $q$, uniformly for moduli $q$ coprime to $6$ that are bounded by a fixed power of $\log{x}$. 
\end{abstract}

\maketitle
 
\section{Introduction}
Let $f$ be an integer-valued arithmetic function. We say $f$ is \textsf{uniformly distributed} (or \textsf{equidistributed}) modulo $q$ if, for each residue class $a\bmod{q}$, 
\[ \#\{n \le x: f(n)\equiv a\pmod{q}\} \sim \frac{x}{q}, \quad\text{as $x\to\infty$}. \]
As a nontrivial example, let  $\Omega(n):=\sum_{p^k\parallel n} k$ be (as usual) the function counting the prime factors of $n$ with multiplicity. Then $\Omega(n)$ is uniformly distributed mod $q$ for every positive integer $q$. This result was first established by Pillai in 1940 \cite{pillai40} but today seems best viewed as a special case of a 1969 theorem of Delange \cite{delange69} characterizing when additive functions are uniformly distributed: An integer-valued additive function $f$ is equidistributed mod $q$, for $q$ odd, if and only if 
\begin{equation}\label{eq:delange1} \sum_{p:~d\nmid f(p)} p^{-1} \quad\text{ diverges}\end{equation}
for every divisor $d>1$ of $q$. If $q$ is even, $f$ is uniformly distributed mod $q$ if and only if (a) \eqref{eq:delange1} holds for every divisor $d>2$ of $q$ \emph{and} (b) either \eqref{eq:delange1} holds when $d=2$, or $f(2^r)$ is odd for every positive integer $r$.

For multiplicative functions, there are indications that uniform distribution is not the correct lens to look through. As a case study, consider Euler's $\phi$-function. It is classical (e.g., implicit in work of Landau \cite{landau09}) that for every $q$, almost all positive integers $n$ are divisible by a prime $p\equiv 1\pmod{q}$. (Here and below, \textsf{almost all} means all numbers $n \le x$ with $o(x)$ exceptions, as $x\to\infty$.) But then $q \mid p-1 \mid \phi(n)$. Thus, $100\%$ of numbers $n$ have $\phi(n)$ belonging to the residue class $0\bmod{q}$, so that equidistribution mod $q$ fails for every $q>1$.  

Motivated by these observations, Narkiewicz in \cite{narkiewicz66} introduces the notion of \textsf{weak uniform distribution}. He calls an integer-valued arithmetic function $f$  \textsf{weakly uniformly distributed} (or \textsf{weakly equidistributed}) modulo $q$ if $\gcd(f(n),q)=1$ for infinitely many $n$ and, for every \emph{coprime} residue class $a\bmod{q}$,
\begin{equation}\label{eq:WUDrelation} \#\{n\le x: f(n)\equiv a\pmod{q}\} \sim \frac{1}{\phi(q)} {\#\{n\le x: \gcd(f(n),q)=1\}} , \quad\text{as $x\to\infty$}. \end{equation}
While $\phi(n)$ is not uniformly distributed modulo any $q>1$, Narkiewicz shows in this same paper that $\phi(n)$ is weakly uniformly distributed modulo $q$ precisely when $\gcd(q,6)=1$. His proof goes by estimating the partial sums of $\chi(\phi(n))$, for Dirichlet characters $\chi$ mod $q$, and depends on the theory of mean values of multiplicative functions built up by Delange and Wirsing.

Various criteria are available to decide weak equidistribution, but it remains a highly nontrivial task to completely determine, for a given $f$, the set of $q$ for which $f$ is weakly equidistributed modulo $q$; see Chapter VI of Narkiewicz's monograph \cite{narkiewicz84} for an algorithmic solution to this problem in certain cases. Of special importance for us is the following partial classification, which is a special case of the main theorem of \cite{narkiewicz82}.

Call an integer-valued multiplicative function $f$ \textsf{polynomially-defined} if for some nonconstant polynomial $F(T) \in \Z[T]$, without multiple roots, we have $f(p)=F(p)$ for all primes $p$. When we refer to $F$ in our results below, we mean the (unique) $F$ associated to $f$ in this way.

\begin{prop} Let $f$ be a polynomially-defined  multiplicative function. There is a constant $C=C(F)$ such that, if $q$ is any positive integer all of whose prime factors exceed $C$, then $f$ is weakly equidistributed modulo $q$.
\end{prop}

In all of the work mentioned so far, the modulus $q$ was assumed to be fixed. It is of some interest to seek uniform versions of these results. Here uniform means that $q$ should be allowed allowed to vary with $x$ (the stopping point of our sample), in analogy with the Siegel--Walfisz theorem from prime number theory. Our first theorem shows that one has uniformity in $q$ up to an arbitrary (but fixed) power of $\log{x}$ when $F$ is linear.

\begin{thm}\label{thm:linear} Let $f$ be a fixed polynomially-defined function with $F(T)=RT+S$, where $R,S\in \Z$ with $R\ne 0$.
Fix a real number  $K>0$. The values $f(n)$, for $n\le x$, are asymptotically weakly uniformly distributed modulo $q$ for all moduli $q\le (\log{x})^{K}$ coprime to $6R$.\footnote{That is, the limit relation \eqref{eq:WUDrelation} holds uniformly in $q$, for these $q$.}
\end{thm}
Thus $\phi(n)$, sampled at numbers $n\le x$, is asymptotically weakly equidistributed mod $q$ uniformly for $q\le (\log{x})^{K}$ with $\gcd(q,6)=1$. We are not sure what to conjecture for how far the range of uniformity can be extended. As discussed in \cite{LPR21},  standard conjectures imply that for $f(n)=\phi(n)$, we cannot replace $(\log{x})^{K}$ with $L(x)^{1+\delta}$ for any $\delta>0$, where $L(x) = x^{ \log\log\log{x}/\log\log{x}}$.

When the defining polynomial $F$ has degree larger than $1$, our method applies but the results require some preparation to state. Let $F(T) \in \Z[T]$ be nonconstant. For each positive integer $q$, define
\begin{equation}\label{eq:nuqdef} \nu(q) = \#\{a\bmod{q}: \gcd(a,q)=1\text{ and } F(a)\equiv 0\pmod{q}\}\end{equation}
and let
\begin{equation}\label{eq:alphaqdef} \alpha(q) = \frac{1}{\phi(q)}\#\{a \bmod{q}: \gcd(aF(a),q)=1\}. \end{equation}
It is straightforward to check, using the Chinese Remainder Theorem, that
\[ \alpha(q) = \prod_{\substack{\ell \mid q \\ \ell \text{ prime}}} \left(1-\frac{\nu(\ell)}{\ell-1}\right). \] 
If $F$ has degree $D$, then $\nu(\ell) \le D$ whenever $\ell$ does not divide the leading coefficient of $F$. Thus, if $q$ is coprime to that coefficient and every prime dividing $q$ exceeds $D+1$, then $\alpha(q)$ is nonzero. Furthermore, by a standard argument with Mertens' theorem, as long as $\alpha(q)$ is nonzero,
\begin{equation}\label{eq:alphaqlower}\alpha(q) \gg_F (\log\log{(3q)})^{-D}. \end{equation}
The lower bound \eqref{eq:alphaqlower} will prove important later.

In what follows, by $\omega(q)$ we shall mean the number of distinct primes dividing $q$. 

\begin{thm}\label{thm:generalF} Let $f$ be a fixed, polynomially-defined multiplicative function. Fix $\delta \in (0, 1]$. There is a constant $C=C(F)$ such that the following holds. For each fixed $K$, the values $f(n)$ for $n\le x$ are asymptotically weakly uniformly distributed mod $q$ provided that $q\le (\log{x})^{K}$, that $q$ is divisible only by primes exceeding $C$, and that either
\begin{enumerate}
    \item[(i)] $q$ is squarefree with 
    $\omega(q) \le (1-\delta)\alpha(q) \log\log{x}/\log{D}$,~\emph{or} 
    \item[(ii)] $q\le (\log{x})^{\alpha(q)(1-\delta)(1-1/D)^{-1}}$.
\end{enumerate}
\end{thm}

Conditions (i) and (ii) in Theorem \ref{thm:generalF} reflect genuine obstructions to uniformity. To motivate (i), fix an integer $D\ge 2$, and let $F(T) = (T-2)(T-4)\cdots (T-2D)+2$. Note that $F$ is Eisenstein at $2$, so $F$ is irreducible over $\Q$ and thus without multiple roots. Let $f$ be the completely multiplicative function with $f(p) = F(p)$ for all primes $p$, and let $q$ be a squarefree product of primes exceeding $D+1$. Then $F(p) \equiv 2\pmod{q}$ whenever $(p-2)\cdots (p-2D)\equiv 0\pmod{q}$. This congruence puts $p$ in one of $D^{\omega(q)}$ coprime residue classes mod $q$. Hence, we expect $\gg\frac{D^{\omega(q)}}{\phi(q)} \frac{x}{\log{x}}$ primes $p\le x$ with $F(p) \equiv 2\pmod{q}$, and we are assured this many primes (by Siegel--Walfisz) if $q$ is bounded by a power of $\log{x}$. On the other hand, Proposition \ref{prop:scourfield} below implies (under this same restriction on the size of $q$) that the number of $n\le x$ with $\gcd(f(n),q)=1$ is  $x/(\log{x})^{1-(1+o(1))\alpha(q)}$. 
Thus, the residue class $2\bmod{q}$ will be `overrepresented' (vis-\`{a}-vis
the expectation of weak uniform distribution) if $D^{\omega(q)} > (\log{x})^{(1+\delta)\alpha(q)}$ for a fixed $\delta>0$, which can happen already with $q \le (\log x)^{O_D(1)}$. \footnote{One can take $q$ to be the product of the primes from $D+1$ up to $K_D \log \log x$, for a suitably chosen constant $K_D$. Here the prime ideal theorem is useful for estimating $\alpha(q)$.}
It follows that (i) is essentially optimal.  

To motivate (ii), fix $D\ge 2$, and let $f$ be the completely multiplicative function given by $f(p) = (p-1)^D+1$ for all primes $p$. Let $q$ be a $D$th power, say $q=q_1^D$. Then $f(p)\equiv 1\pmod{q}$ whenever $p\equiv 1\pmod{q_1}$. Thus, if $q$ is bounded by a power of $\log{x}$, there will be $\gg x/\phi(q_1)\log{x}$ primes $p\le x$ for which $f(p)\equiv 1\pmod{q}$. On the other hand, if we assume all primes dividing $q_1$ exceed $D+1$, Proposition  \ref{prop:scourfield} implies that there are $x/(\log{x})^{1-(1+o(1))\alpha(q)}$ integers $n\le x$ with $\gcd(f(n),q)=1$. It follows that the residue class $1\bmod{q}$ will be overrepresented if $q^{1-1/D}= q/q_1 > (\log{x})^{(1+\delta)\alpha(q)}$. 
This means that for weak equidistribution we require $q$ to be no more than $\approx (\log{x})^{\alpha(q)(1-1/D)^{-1}}$.
So (ii) is essentially best possible as well. 

In both of the constructions described above, the obstruction to uniformity came from prime inputs $p$. Tweaking the construction slightly, we could easily produce obstructions to uniformity of the form $rp$, with $r$ fixed (or even with $r$ growing slowly with $x$). In our final theorem, we pinpoint the `problem' here as one of having too few large prime factors. Specifically, we show that uniformity up to an arbitrary power of $\log{x}$ can be restored by considering only inputs with sufficiently many prime factors exceeding $q$. In fact, for squarefree moduli $q$, it suffices to restrict to inputs with composite $q$-rough part. 

We write $P(n)$ for the largest prime factor of $n$, with the convention that $P(1)=1$. We set $P_1(n)=P(n)$ and define, inductively, $P_k(n) = P_{k-1}(n/P(n))$. Thus, $P_k(n)$ is the $k$th largest prime factor of $n$, with $P_k(n)=1$ if $\Omega(n) < k$.

\begin{thm}\label{thm:severallarge} Let $f$ be a fixed, polynomially-defined function. There is a constant $C(F)$ such that the following hold. 
\begin{enumerate}
\item[(a)] For each fixed $K>0$, 
\begin{multline}\label{eq:severallargeeq} \#\{n\le x: P_{D+2}(n) > q,~f(n)\equiv a\pmod{q}\} \\
\sim \frac{1}{\phi(q)}{\#\{n\le x: P_{D+2}(n)>q,~\gcd(f(n),q)=1\}} \quad\text{ as $x\to\infty$}, \end{multline}
uniformly for coprime residue classes $a\bmod{q}$ with $q\le (\log{x})^{K}$ and $q$ divisible only by primes exceeding $C(F)$. 
\item[(b)] For each fixed $K>0$,
\begin{multline*} \#\{n\le x: P_{2}(n) > q,~f(n)\equiv a\pmod{q}\} \\
\sim \frac{1}{\phi(q)}{\#\{n\le x: P_{2}(n)>q,~\gcd(f(n),q)=1\}} \quad\text{ as $x\to\infty$}, \end{multline*}
uniformly for coprime residue classes $a\bmod{q}$ with $q$ squarefree, $q\le (\log{x})^{K}$, and $q$ divisible only by primes exceeding $C(F)$. 
\end{enumerate}
\end{thm}

The method of the present paper refines that of the authors' earlier works \cite{LPR21, PSR22}. In those papers, it was crucial that the modulus $q$ be either prime or `nearly prime', in the sense that $\sum_{\ell \mid q} 1/\ell = o(1)$. The essential new ingredient here, which allows us to dispense with any such condition, is the exploitation of a certain  ergodic (or mixing) phenomenon within the multiplicative group mod $q$. As one illustration: Let $q$ be a positive integer coprime to $6$. From the collection of units $u$ mod $q$ for which $u+1$ is also a unit, choose uniformly at random $u_1, u_2, u_3,\dots$, and construct the products $u_1, u_1 u_2, u_1 u_2 u_3, \dots$. Once $J$ is large, each unit mod $q$ is roughly equally likely to appear as $u_1 \cdots u_J$. This particular example plays a starring role in our approach to the weak equidistribution of Euler's $\phi$-function. 

When $f=\phi$, Theorem \ref{thm:linear} is in the spirit of the Siegel--Walfisz theorem, with primes replaced by  values of $\phi(n)$. For investigations of the corresponding `Linnik's theorem', concerning the least $n$ for which $\phi(n)$ falls into a given progression, see \cite{CG09,FL08, FS07,garaev09}.

Finally, it is worth mentioning that although in the spirit of Narkiewicz's results, we stated Theorems \ref{thm:linear}, \ref{thm:generalF} and \ref{thm:severallarge} for $F(T) \in \Z[T]$, our methods go through (with minor modifications) for integer-valued polynomials $F$, namely those satisfying $F(\Z) \subset \Z$. Writing any such polynomial in the form $G(T)/Q$ for some positive integer $Q$ 
and 
$G(T) \in \Z[T]$, we need only ensure in addition that the constant $C(F)$ appearing in the aforementioned theorems exceeds $Q$. 

\subsection*{Notation and conventions} We do not consider the zero function as   multiplicative (thus, if $f$ is multiplicative, then $f(1)=1$). Throughout, the letters $p$ and $\ell$ are to be understood as denoting primes. Implied constants in $\ll$ and $O$-notation may always depend on any parameters  declared as ``fixed''; other dependence will be noted explicitly (for example, with subscripts). We use $\log_{k}$ for the $k$th iterate of the natural logarithm. When there is no danger of confusion, we write $(a,b)$ instead of $\gcd(a,b)$. 

\section{A preparatory estimate: The frequency with which $(f(n),q)=1$}

The following proposition is contained in results of Scourfield \cite{scourfield84}. Nevertheless, we give a complete treatment here, for two reasons. First, we prefer to keep matters as self-contained as possible. Second, the results of \cite{scourfield84} are much more precise than we will need. The weaker version below admits a simpler and shorter proof (although we make no claim to originality regarding the underlying ideas).

For readability, we sometimes abbreviate $\alpha(q)$ to $\alpha$, suppressing the dependence on $q$.

\begin{prop}\label{prop:scourfield} Fix a multiplicative function $f$ with the property that $f(p)=F(p)$ for all primes $p$, where $F(T) \in \Z[T]$ is nonconstant. Fix $K>0$. If $x$ is sufficiently large and $q\le (\log{x})^{K}$ with $\alpha = \alpha(q)> 0$, then
\begin{equation}\label{eq:scourfieldestimate} \#\{n\le x: (f(n),q)=1\} = \frac{x}{(\log{x})^{1-\alpha}} \exp(O((\log\log{(3q)})^{O(1)})). \end{equation}
\end{prop}

We treat separately the implicit upper and lower bounds in Proposition \ref{prop:scourfield}.

\subsection*{Upper bound} 
The following mean value estimate is a simple consequence of \cite[Theorem 01, p.\ 2]{HT88} (and also of the more complicated Theorem 03 from that same chapter).
\begin{lem}\label{lem:HT} Let $g$ be a  multiplicative function with $0\le g(n)\le 1$ for all $n$. For all $x\ge 3$, 
\[ \sum_{n \le x} g(n) \ll \frac{x}{\log{x}} \exp\left(\sum_{p \le x}\frac{g(p)}{p} \right). \]
Here the implied constant is absolute.
\end{lem}

If we set $g(n):= \1_{\gcd(f(n),q)=1}$, then the left-hand side of \eqref{eq:scourfieldestimate} is precisely $\sum_{n \le x} g(n)$. Note that the multiplicativity of $f$ implies the multiplicativity of $g$. The following lemma, due independently to Norton \cite[Lemma, p.\
669]{norton76} and Pomerance \cite[Remark 1]{pomerance77}, allows us to estimate the sums of $g(p)/p$ appearing in Lemma \ref{lem:HT}.

\begin{lem}\label{lem:NP} Let $q$ be a positive integer, and suppose $x$ is a real number with $x\ge \max\{3,q\}$. For each coprime residue class $a\bmod{q}$,
\[ \sum_{\substack{p \le x \\ p \equiv a\pmod{q}}} \frac{1}{p} = \frac{\log_2{x}}{\phi(q)} + \frac{1}{p_{q,a}} + O\left(\frac{\log{(3q)}}{\phi(q)}\right), \] where $p_{q,a}$ denotes the least prime congruent to $a$ modulo $q$.
\end{lem}

\begin{lem}\label{lem:primesum} Let $F(T)\in \Z[T]$ be a fixed nonconstant polynomial. For each positive integer $q$ and each real number $x\ge 3q$, \[ \sum_{p \le x} \frac{\1_{\gcd(F(p),q)=1}}{p} = \alpha \log_2{x} + O((\log\log{(3q)})^{O(1)}),  \]
where $\alpha=\alpha(q)$ is as defined in \eqref{eq:alphaqdef}.
\end{lem} 

\begin{proof} Using the M\"{o}bius function to detect the coprimality condition, we write
\begin{align} \sum_{\substack{p \le x \\ \gcd(F(p),q)=1}} \frac{1}{p} &= \sum_{\substack{3q < p \le x \notag \\ \gcd(F(p),q)=1}} \frac{1}{p} + O(\log_{2}(100q)) \\
&= \sum_{d \mid q} \mu(d) \sum_{\substack{3q < p \le x \\ d\mid F(p)}} \frac{1}{p} + O(\log_{2}(100q)).\label{eq:primesum2}
\end{align}
If $p$ is a prime with $p > 3q$, then $d\mid F(p)$ precisely when $p$ belongs to one of $\nu(d)$ coprime residue classes modulo $d$. By Lemma \ref{lem:NP} (with $d$ replacing $q$), \[ \sum_{\substack{3q < p \le x \\ d\mid F(p)}} \frac{1}{p} = \frac{\nu(d)}{\phi(d)} \log\log{x} + O\left(\frac{\nu(d)\log(3d)}{\phi(d)} + \frac{\nu(d) \log_2(3q)}{\phi(d)}\right). \]  
Substituting this estimate into \eqref{eq:primesum2} yields a main term of $(\sum_{d\mid q} \frac{\mu(d)\nu(d)}{\phi(d)})\log_2{x} = \alpha \log_2{x}$, as desired. Turning to the errors,
\begin{align*} \sum_{\substack{d \mid q\\d\text{ squarefree}}}  \frac{\nu(d) \log(3d)}{\phi(d)} &= \sum_{\substack{d \mid q \\ d\text{ squarefree}}} \frac{\nu(d)}{\phi(d)} (\log{3}+\sum_{\ell \mid d} \log(\ell)) 
\\
&\le (\log{3})\sum_{\substack{d \mid q \\ d\text{ squarefree}}}\frac{\nu(d)}{\phi(d)}+\sum_{\ell \mid q} \log{\ell} \cdot \frac{\nu(\ell)}{\ell-1} \sum_{\substack{r \mid q/\ell \\r \text{ squarefree}}}\frac{\nu(r)}{\phi(r)} \\
&\ll \bigg(\sum_{\substack{d \mid q \\ d\text{ squarefree}}} \frac{\nu(d)}{\phi(d)}\bigg)\bigg(1+ \sum_{\ell \mid q}\log{\ell}\cdot \frac{\nu(\ell)}{\ell-1}\bigg).\end{align*}
Now $\sum_{d \mid q,~d\text{ squarefree}} \frac{\nu(d)}{\phi(d)} = \prod_{\ell \mid q}(1+\nu(\ell)/(\ell-1)) \ll (\log_{2}(3q))^{D}$ 
(keeping in mind that $\nu(\ell)\le D$ for all but $O(1)$
many primes $\ell$). Furthermore,
\[ \sum_{\ell \mid q} \nu(\ell)\frac{\log{\ell}}{\ell-1} \ll \sum_{\ell \mid q} \frac{\log{\ell}}{\ell} \le \sum_{\ell \le \log{(3q)}} \frac{\log{\ell}}{\ell} + \sum_{\substack{\ell \mid q \\ \ell > \log(3q)}} \frac{\log{\ell}}{\ell} \ll \log_2{(3q)} + \frac{\log_2{(3q)}}{\log{(3q)}}\sum_{\substack{\ell\mid q \\ \ell > \log(3q)}} 1, \]
and this is 
\[ \ll \log_2{(3q)} + \frac{\log_2{(3q)}}{\log(3q)} \cdot \frac{\log{q}}{\log_2{(3q)}} \ll \log_2{(3q)}. \]
Thus, $\sum_{d \mid q,~d\text{ squarefree}} \frac{\nu(d)\log{(3d)}}{\phi(d)} \ll (\log_{2}(3q))^{D+1}$. Finally, 
\[ \sum_{\substack{d \mid q \\ d\text{ squarefree}}} \frac{\nu(d) \log_2{(3q)}}{\phi(d)} \ll \log_2{(3q)} \cdot \prod_{\ell \mid q} \left(1 + \frac{\nu(\ell)}{\ell-1} \right)\ll (\log_2{(3q)})^{D+1}. \]
Collecting estimates, $\sum_{p\le x} \1_{\gcd(F(p),q)=1}/p = \alpha\log_2{x} + O((\log_2{(3q)})^{D+1})$. 
\end{proof}

The upper bound half of Proposition \ref{prop:scourfield} follows (in slightly more precise form) immediately from Lemmas \ref{lem:HT} and \ref{lem:primesum}. In fact, we have shown the upper bound in the much wider range $q\le x/3$.

\subsection*{Lower bound}
The following lemma is due to Barban \cite[Lemma 3.5]{barban66}; see also \cite[Theorem 3.5, p.\ 61]{SS94}.

\begin{lem}\label{lem:barban} Let $g$ be a multiplicative function with $0 \le g(n) \le 1$ for all $n$. For all $x\ge 3$,
\[ \sum_{\substack{n \le x \\n \text{ squarefree}}} \frac{g(n)}{n} \gg \exp\left(\sum_{p \le x} \frac{g(p)}{p}\right). \]
Here the implied constant is absolute.
\end{lem}

\begin{proof}[Proof of the lower bound in Proposition \ref{prop:scourfield}] Consider $n$ of the form $mP$, where $m\le x^{1/3}$ is a squarefree product of primes $p$ with $\gcd(f(p),q)=1$ and $P \in (x^{1/2},x/m]$ is a prime with $(f(P),q)=1$. Each such $n$ has $f(n) = f(m)f(P)$ coprime to $q$. 

Given $m$ as above, we count corresponding $P$\!. The prime $P$ is restricted to one of the $\alpha(q) \phi(q)$ residue classes $a$ mod $q$ with $\gcd(aF(a),q)=1$. Hence, given $m\le x^{1/3}$ as above, the Siegel--Walfisz theorem guarantees that there are \[ \gg (\alpha(q)\phi(q)) \cdot \frac{1}{\phi(q)} \frac{x}{m\log{x}} = \alpha(q) \frac{x}{m\log{x}} \] values of $P$. Now sum on $m$; by Lemma \ref{lem:barban},
\[ \sum \frac{1}{m} = \sum_{\substack{m \le x^{1/3}\\m~\text{squarefree}}} \frac{\1_{\gcd(f(m),q)=1}}{m}
\gg \exp\left(\sum_{p \le x^{1/3}} \frac{\1_{\gcd(f(p),q)=1}}{p}\right). \]
The final sum on $p$ is within $O(1)$ of the corresponding sum taken over all $p \le x$. The lower bound half of Proposition \ref{prop:scourfield} now follows from Lemma \ref{lem:primesum}, bearing in mind that $\alpha(q) \gg (\log\log{(3q)})^{-D}$.
\end{proof}

\section{Framework for the proof of Theorems \ref{thm:generalF} and \ref{thm:severallarge}} \label{sec:framework}
Define $J=J(x)$ by setting
\[ J=\lfloor \log\log\log{x}\rfloor. \]
(For our purposes, any integer-valued function tending to infinity sufficiently slowly would suffice.) With $\delta$ from the statement of Theorem \ref{thm:generalF}, we let $y=y(x)$ be defined by
\[ y := \exp((\log{x})^{\delta/2}) 
\] 
and we say that the positive integer $n$ is \textsf{convenient} (with respect to a given large real number $x$) if (a) $n\le x$, (b) the $J$ largest prime factors of $n$ exceed $y$, and (c) none of these $J$ primes are repeated in $n$. That is, $n$ is convenient if $n$ admits an expression $n=mP_J\cdots P_1$, where $P_1, \dots, P_J$ are primes with
\begin{equation}\label{eq:primesrestriction1} \max\{P(m),y\} < P_J < \dots < P_1, \end{equation}
\begin{equation}\label{eq:primesrestriction2} P_J \cdots P_1 \le x/m. \end{equation}
The framework developed in this section will go through in the proof of Theorem \ref{thm:severallarge} (\S\ref{sec:severallarge}) by setting $\delta:=1$.   

Now let $f$ be a fixed multiplicative function with $f(p)=F(p)$ for all primes $p$, where $F(T)\in \Z[T]$ is nonconstant. Fix $K>0$, and suppose that $q\le (\log{x})^{K}$. We set
\[ N(q) = \#\{n\le x: \gcd(f(n),q)=1\}, \]
and we define $N_\con(q)$ and $N_{\inc}(q)$ analogously, incorporating the extra requirement that $n$ be convenient or inconvenient, respectively.

\begin{lem}\label{lem:nqvsnqcon} $N(q)\sim N_{\con}(q)$, as $x\to\infty$. Here the asymptotic holds uniformly in $q$ with $q\le (\log{x})^{K}$ and $\alpha(q)\ne 0$. \end{lem}

\begin{proof} We must show that $N_{\inc}(q) = o(N(q))$, as $x\to\infty$.

Suppose the integer $n\le x$ is counted by $N_{\inc}(q)$. We can assume that $P(n) > z:= x^{1/\log_2{x}}$. Indeed, by well-known results on smooth numbers (for instance \cite[Theorem 5.13\text{ and }Corollary 5.19, Chapter III.5]{tenenbaum15}), 
the number of $n\le x$ with $P(n) \le z$ is at most $x/(\log{x})^{(1+o(1))\log_3 x}$ 
and this is $o(N(q))$ by our `rough-and-ready' estimate of Proposition \ref{prop:scourfield}. We can similarly assume that $n$ has no repeated prime factors exceeding $y$, since the number of exceptions is $O(x/y)$, which is again $o(N(q))$.

Write $n=PAB$, where $P= P(n)$ and $A$ is the largest divisor of $n/P$ supported on primes exceeding $y$.  Observe that $AB = n/P \le x/z$. So if $A$ and $B$ are given, the number of possibilities for $P$ is bounded by $\pi(x/AB) \ll x/AB \log{z} \ll x(\log\log{x})/AB\log{x}$. 
We sum on $A,B$. As $n$ has no repeated primes exceeding $y$ but $n$ is inconvenient, it must be that $\Omega(A) < J$.  Thus, $\sum 1/A \le (1+ \sum_{p \le x} 1/p)^{J} \le (2\log_2{x})^{J} \le \exp(O((\log_3 x)^2))$. Using that $(f(B),q)=1$ (as $f(n) = f(B) f(AP)$) and that $B$ is $y$-smooth, 
\[ \sum \frac{1}{B} \le \prod_{p \le y} \left(\sum_{j=0}^{\infty} \frac{\1_{(f(p^j),q)=1}}{p^j}\right)\ll \exp\left(\sum_{p\le y} \frac{\1_{(f(p),q)=1}}{p}\right), \]
and this is $\ll (\log{x})^{\alpha\delta/2} \exp(O((\log_2 q)^{O(1)}))$ by Lemma \ref{lem:primesum}. 
We conclude that these $n$ make a contribution to $N_{\inc}(q)$ of size at most $\frac{x}{(\log{x})^{1-\alpha\delta/2}} \exp(O((\log_3{x})^2 + (\log_2 q)^{O(1)}))$. 
Since  $q\le (\log{x})^{K}$ and $\alpha(q)$ obeys the lower bound \eqref{eq:alphaqlower}, this contribution is also $o(N(q))$. 
\end{proof}

Let $N(q,a)$ denote the number of $n \le x$ with $f(n)\equiv a\pmod{q}$, and define  $N_{\con}(q,a)$ and $N_{\inc}(q,a)$ analogously. By Lemma \ref{lem:nqvsnqcon}, the weak equidistribution of $f$ mod $q$ will follow if $N(q,a) \sim \frac{1}{\phi(q)} N_{\con}(q)$. 

As a first step in this direction, we compare $N_{\con}(q)$ and $N_{\con}(q,a)$.  Clearly,
\[ N_{\con}(q) = \sum_{\substack{m \le x \\ \gcd(f(m),q)=1}} \sideset{}{'}\sum_{P_1, \dots, P_J} 1,\]
where the $'$ on the sum indicates that $P_1,\dots,P_J$ run through primes satisfying \eqref{eq:primesrestriction1}, \eqref{eq:primesrestriction2},
and
\begin{equation}\label{eq:primesrestriction3} \gcd(f(P_1) \cdots f(P_J),q)=1.
\end{equation}
Similarly,
\[ N_{\con}(q,a) = \sum_{\substack{m \le x \\ \gcd(f(m),q)=1}} \sideset{}{''}\sum_{P_1, \dots, P_J} 1, \]
where the $''$ condition indicates that we enforce \eqref{eq:primesrestriction1},  \eqref{eq:primesrestriction2} and (in place of \eqref{eq:primesrestriction3})
\begin{equation}\label{eq:primesrestriction4} f(m) f(P_1) f(P_2) \cdots f(P_J) \equiv a\pmod{q}. \end{equation}

Let 
\[ \V_q' = \{(v_1,\dots,v_J)\bmod{q}: \gcd(v_1 \dots v_J,q)=1, \gcd(F(v_1)\cdots F(v_J),q)=1\}  \] and
\[ \V_{q, a, m}'' = \{(v_1, \dots, v_J)\bmod{q}: \gcd(v_1 \dots v_J,q)=1, f(m) F(v_1)\cdots F(v_J) \equiv a\pmod{q}\}.  \]
Then \eqref{eq:primesrestriction3} amounts to restricting $(P_1, \dots, P_J)$, taken mod $q$, to belong to $\V_q'$, while \eqref{eq:primesrestriction4} restricts this same tuple to $\V_{q, a, m}''$. By \eqref{eq:alphaqdef}, $\#\V_q' = (\phi(q)\alpha(q))^{J}$. 

The conditions \eqref{eq:primesrestriction2} and \eqref{eq:primesrestriction3} are independent of the ordering of $P_1,\dots, P_J$. Thus, letting $L_m = \max\{y, P(m)\}$, \begin{equation}\label{eq:toremovecongruences}  \sideset{}{'}\sum_{P_1, \dots, P_J} 1 = \frac{1}{J!} \sum_{\mathbf{v} \in \V_q'}  \sum_{\substack{P_1,\dots, P_J\text{ distinct} \\ P_1\cdots P_J \le x/m \\ \text{each } P_j > L_m \\ \text{each } P_j \equiv v_j\pmod{q}}} 1. \end{equation} 

We proceed to remove the congruence conditions on the $P_j$ from the inner sum. For each tuple $(v_1,\dots, v_J)\bmod{q} \in \V_q'$,
\[ \sum_{\substack{P_1,\dots, P_J\text{ distinct} \\ P_1\cdots P_J \le x/m \\ \text{each } P_j > L_m \\ \text{each } P_j \equiv v_j\pmod{q}}} 1 =  \sum_{\substack{P_2,\dots, P_J\text{ distinct} \\ P_2\cdots P_J \le x/m L_m \\ \text{each } P_j > L_m \\ \text{each } P_j \equiv v_j\pmod{q}}} \sum_{\substack{P_1 \neq P_2,\dots,P_J \\ L_m < P_1 \le x/mP_2\cdots P_J \\ P_1 \equiv v_1\pmod{q}}} 1.\]

Since $L_m \ge y$ and $q\le (\log{x})^{K} = (\log{y})^{2K/\delta}$, the Siegel--Walfisz theorem implies that 
\[ \sum_{\substack{P_1 \neq P_2,\dots,P_J \\ L_m < P_1 \le x/mP_2\cdots P_J \\ P_1 \equiv v_1\pmod{q}}} 1 = \frac{1}{\phi(q)} \sum_{\substack{P_1 \neq P_2,\dots,P_J \\ L_m < P_1 \le x/mP_2\cdots P_J}}1 + O\left(\frac{x}{mP_2\cdots P_J} \exp(-C_0 \sqrt{\log y})\right),\] 
for some positive constant $C_0 := C_0(K, \delta)$ depending only on $K$ and $\delta$. Putting this back into the last display and bounding the $O$-terms crudely, we find that
\[ \sum_{\substack{P_1,\dots, P_J\text{ distinct} \\ P_1\cdots P_J \le x/m \\ \text{each } P_j > L_m \\ \text{each } P_j \equiv v_j\pmod{q}}} 1 =  \frac{1}{\phi(q)}\sum_{\substack{P_1,\dots, P_J\text{ distinct} \\ P_1\cdots P_J \le x/m \\ \text{each } P_j > L_m \\ (\forall j\ge 2)~P_j \equiv v_j\pmod{q}}} 1 + O\left(\frac{x}{m} \exp\left(-\frac{1}{2} C_0 (\log{x})^{\delta/4}\right)\right).  \] 
Proceeding in the same way to remove the congruence conditions on $P_2,\dots,P_J$, we arrive at the estimate
\[ \sum_{\substack{P_1,\dots, P_J\text{ distinct} \\ P_1\cdots P_J \le x/m \\ \text{each } P_j > L_m \\ \text{each } P_j \equiv v_j\pmod{q}}} 1 =  \frac{1}{\phi(q)^{J}}\sum_{\substack{P_1,\dots, P_J\text{ distinct} \\ P_1\cdots P_J \le x/m \\ \text{each } P_j > L_m}} 1 + O\left(\frac{x}{m} \exp\left(-\frac{1}{4} C_0 (\log{x})^{\delta/4}\right)\right).\] 
Inserting this estimate into \eqref{eq:toremovecongruences} and keeping in mind that $\#\V_q' \le (\log{x})^{KJ}$ (trivially), 
we conclude that
\begin{align}\notag N_{\con}(q)&= \sum_{\substack{m \le x \\ \gcd(f(m),q)=1}} \sideset{}{'}\sum_{P_1, \dots, P_J} 1\\ 
&= \sum_{\substack{m\le x \\ \gcd(f(m),q)=1}}\frac{\#\V_q'}{\phi(q)^{J}}\Bigg(\frac{1}{J!} \sum_{\substack{P_1,\dots, P_J\text{ distinct} \\ P_1\cdots P_J \le x/m \\ \text{each } P_j > L_m}} 1\Bigg) + O\left(x \exp\left(-\frac{1}{8} C_0 (\log{x})^{\delta/4}\right)
\right). \label{eq:hasanalogue} 
\end{align}

An entirely analogous argument yields the same estimate with $N_{\con}(q)$ replaced by $N_{\con}(q,a)$ and $\V_q'$ replaced by $\V_{q, a, m}''$. Comparing \eqref{eq:hasanalogue} with its $N_{\con}(q,a)$ analogue and rewriting \[ \frac{\#\V_{q, a, m}''}{\phi(q)^J} =  \frac{\#\V_{q, a, m}''}{\#\V_q'} \cdot \frac{\#\V_q'}{\phi(q)^J},\]  we are motivated to introduce the following hypothesis.

\begin{hypa} $\frac{\#\V_{q, a, m}''}{\#\V_q'} \sim \frac{1}{\phi(q)}$, as $x\to\infty$, uniformly in $q$ and $a$ and uniformly in $m\le x$ with $\gcd(f(m),q)=1$. 
\end{hypa}
We will soon see how to verify Hypothesis A in the situations described in Theorems \ref{thm:linear}, \ref{thm:generalF}, and \ref{thm:severallarge}. The phrase ``uniformly in $q$ and $a$'' in Hypothesis A should be read as ``uniformly in $q$ and $a$ subject to the restrictions of these theorem statements''. 

If Hypothesis A holds, we may deduce (keeping in mind Lemma \ref{lem:nqvsnqcon}, and that $x\exp(-\frac{1}{8}C_0 (\log{x})^{\delta/4}) = o(N(q)/\phi(q))$)  
\begin{align*} N_{\con}(q,a)
&= \sum_{\substack{m \le x \\ \gcd(f(m),q)=1}} \sideset{}{''}\sum_{P_1, \dots, P_J} 1\\ 
&= (1+o(1)) \frac{1}{\phi(q)} N_{\con}(q) + o\left( \frac{N(q)}{\phi(q)}\right) 
= (1+o(1)) \frac{1}{\phi(q)} N(q).\end{align*}
Since $N(q,a) = N_{\con}(q,a) + N_{\inc}(q,a)$,  weak uniform distribution mod $q$ will follow if the contribution from $N_{\inc}(q,a)$ is shown to be negligible. We record this condition as our next Hypothesis. 

\begin{hypb} $N_{{\rm inc}}(q,a) = o(N(q)/\phi(q))$, as $x\to\infty$, uniformly in $q$ and $a$. 
\end{hypb}

\section{Linearly defined functions: Proof of Theorem \ref{thm:linear}}\label{sec:linear}
We proceed to verify Hypotheses A and B.

\begin{proof}[Verification of Hypothesis A] Let $m\le x$ with $\gcd(f(m),q)=1$, and let $w \in \Z$ be a value of $af(m)^{-1}$ modulo $q$. We will estimate $\#\V_{q, a, m}''$ via the product formula $\#\V_{q, a, m}'' = \prod_{\ell^{e} \parallel q} V''_{\ell^e}$, where
\[ V''_{\ell^e} := \#\{(v_1, \dots, v_{J}) \bmod{\ell^e}: \gcd(v_1 \dots v_J,\ell)=1,~ \prod_{i=1}^{J}(Rv_i+S)\equiv w\pmod{\ell^e}\}. \] 
By assumption, $(\ell,6R)=1$ for all $\ell\mid q$. 

Suppose first that $\ell\mid S$.
Then the condition $\gcd(v_1 \dots v_J,\ell)=1$ is implied by $\prod_{i=1}^{J}(Rv_i+S)\equiv w\pmod{\ell^e}$. Noting that the map $v \mapsto Rv+S$ is a permutation of $\Z/\ell^e\Z$, we see that $V''_{\ell^e} = \phi(\ell^e)^{J-1}$ and
\begin{equation}\label{eq:whenpmidT} \phi(\ell^e) V''_{\ell^e} = \phi(\ell^e)^{J}.\end{equation}

When $\ell\nmid S$, we must work somewhat harder. By inclusion-exclusion, \begin{equation}\label{eq:incex} V''_{\ell^e} = \sum_{j=0}^{J} (-1)^{j} \binom{J}{j}  V''_{\ell^e,j},\end{equation} where \[ V''_{\ell^e,j} = \#\{(v_1,\dots,v_J)\bmod{\ell^e}: \ell\mid v_1, v_2, \dots, v_j,~\prod_{i=1}^{J}(Rv_i+S) \equiv w\pmod{\ell^e}\}. \] 
If $0 \le j < J$, then $V''_{\ell^e,j} =(\ell^{e-1})^{j} \phi(\ell^e)^{J-j-1}$: Each of $v_1,\dots,v_j$ can be chosen arbitrarily from the $\ell^{e-1}$ classes divisible by $\ell$, 
while  $v_{j+1},\dots,v_{J-1}$ can be chosen arbitrarily subject to each of $Rv_i+S$ (for $i=j+1,\dots,J-1$) being a unit mod $\ell^e$; this then determines $v_J$. Similarly, $V''_{\ell^e,J} = O((\ell^{e-1})^{J-1})$. Referring back to \eqref{eq:incex}, \begin{align}\notag \phi(\ell^e) V''_{\ell^e} &= (\phi(\ell^e)-\ell^{e-1})^{J} + O(\ell^e (\ell^{e-1})^{J-1}) \\
&=(\ell^{e} (1-2/\ell))^J\left(1+O(\ell(\ell-2)^{-J})\right). \label{eq:whenpnmidT}\end{align}
By \eqref{eq:whenpmidT} and \eqref{eq:whenpnmidT}, in either case for $\ell$ we have
\[ \phi(\ell^e) V''_{\ell^e} = \left(\phi(\ell^e) 
\left(1-\frac{\nu(\ell)}{\ell-1}\right)\right)^{J} \cdot \left(1+O(\ell(\ell-2)^{-J})\right). \]
Multiplying over $\ell$, 
\[ \phi(q) \#\V_{q, a, m}'' = (\phi(q) \alpha(q))^{J} \prod_{\ell^e \parallel q} \left(1+O(\ell(\ell-2)^{-J})\right) = \#\V_q' \prod_{\ell^e \parallel q} \left(1+O(\ell(\ell-2)^{-J})\right). \] So to verify Hypothesis A, it is enough to show that the final product is $1+o(1)$. This follows if $\sum_{\ell^e\parallel q} \ell(\ell-2)^{-J} = o(1)$, which is straightforward to prove: Since $q$ is coprime to $6$, we have for all large $x$ that
\[ \sum_{\ell^e\parallel q} \ell(\ell-2)^{-J} < \sum_{\ell \ge 5} \ell(\ell-2)^{-J} \le 3^{-J/2} \sum_{\ell \ge 5} \ell(\ell-2)^{-J/2} \le 3^{-J/2} \sum_{\ell \ge 5} \ell (\ell-2)^{-3} \ll 3^{-J/2}. \qedhere \]
\end{proof}

\begin{rmk}
It is also possible to estimate $V''_{\ell^e}$ via character sums, which will be our primary tool for general $F(T) \in \Z[T]$. By orthogonality (as in \eqref{eq:orthogonalityapplication0} below), $\phi(\ell^e) V''_{\ell^e} = \sum_{\chi \bmod{\ell^e}} \bar\chi(w) Z_\chi^J$, where 
\begin{align*} Z_\chi :&= \sum_{v \bmod \ell^e} \chi_0(v) \chi(Rv+S) \\ &= \sum_{u \bmod{\ell^e}} \chi(u) - \sum_{\substack{u \bmod{\ell^e}\\u \equiv S \bmod \ell}} \chi(u); \end{align*} here we have used that as $v$ runs over coprime residues mod $\ell^e$, the expression $Rv+S$ runs over all the residues mod $\ell^e$ except for those congruent to $S$ mod $\ell$. If $\ell\mid S$, it is then immediate that $Z_\chi = \1_{\chi = \chi_0} \phi(\ell^e)$ (with $\chi_0$ denoting the principal character mod $\ell^e$), once again giving us $\phi(\ell^e) V''_{\ell^e} = \phi(\ell^e)^J$. On the other hand, if $\ell \nmid S$, then fixing a generator $g$ mod $\ell^e$ and considering the unique $r \in \{0, 1, \dots, \phi(\ell^e)-1\}$ satisfying $g^r \equiv S \pmod{\ell^e}$, we observe that the sets $\{u \bmod{\ell^e}: u \equiv S \bmod \ell\}$ and $\{g^{r+(\ell-1)k} \bmod{\ell^e}: 0 \leq k<\ell^{e-1}\}$ are equal. Hence, \[ \sum_{\substack{u \bmod{\ell^e}\\u \equiv S \bmod \ell}} \chi(u) = \1_{\chi^{\ell-1} = \chi_0} \chi(S) \ell^{e-1}.\] As such, $Z_\chi =
\1_{\chi=\chi_0} \ell^{e-1}(\ell-2) +O(\1_{\chi^{\ell-1}=\chi_0,~\chi\ne \chi_0} \ell^{e-1})$, which again leads to (\ref{eq:whenpnmidT}) since there are $\ell-2$ nontrivial characters $\chi$ mod $\ell^e$ satisfying $\chi^{\ell-1} = \chi_0$.
\end{rmk}

\begin{proof}[Verification of Hypothesis B] We proceed as in the proof of Lemma \ref{lem:nqvsnqcon}. Let $n\le x$ be an inconvenient solution to $f(n)\equiv a\pmod{q}$. We can assume $P(n) > z= x^{1/\log_2 x}$, since the number of exceptional $n\le x$ is  $o(N(q)/\phi(q))$. Similarly, we can assume that $n$ has no repeated prime factors exceeding $y=\exp((\log{x})^{\delta/2})$. 
Write $n=PAB$, where $P:=P(n)$ and $A$ is the largest divisor of $n/P$ supported on primes exceeding $y$. Then $z < P\le x/AB$ and $(RP+S)f(AB)\equiv a\pmod{q}$. Given $A$ and $B$, this congruence is satisfied for $P$ belonging to at most one coprime residue class mod $q$. So by the Brun--Titchmarsh inequality, given $A$ and $B$ there are $\ll x/\phi(q) AB \log{(z/q)} \ll x\log_2{x}/\phi(q) AB \log x$ 
corresponding values of $P$. Note that we have saved a factor of $\phi(q)$ here over the analogous estimate in Lemma \ref{lem:nqvsnqcon}. Summing on $A,B$, and making the same estimates as in the argument for Lemma \ref{lem:nqvsnqcon}, yields \[ N_{\inc}(q,a) \le \frac{x}{\phi(q) (\log{x})^{1-\alpha\delta/2}} \exp(O((\log_3{x})^2 + (\log_2 q)^{O(1)})), \] 
and this is $o(N(q)/\phi(q))$.
\end{proof}

\section{General polynomially defined functions: Proof of Theorem \ref{thm:generalF}}\label{sec:general}

To check Hypothesis A in the context of Theorem \ref{thm:generalF}, we require the following character sum estimate, which follows from the Weil bounds when $e=1$ and from work of Cochrane \cite{cochrane02} (see also \cite{cochrane03}) when $e>1$. See \cite[Proposition 2.6]{PSR22} for a detailed discussion. 

\begin{lem}\label{lem:combinedcharsum} Let $F_1(T)$, \dots, $F_K(T) \in \Z[T]$ be nonconstant polynomials for which the product $F_1(T) \cdots F_K(T)$ has no multiple roots. Let $\ell$ be an odd prime not dividing the leading coefficient of any of the $F_k(T)$ and not dividing the discriminant of $F_1(T) \cdots F_K(T)$. Let $e$ be a positive integer, and let $\chi_1,\dots,\chi_K$ be Dirichlet characters modulo $\ell^e$, at least one of which is primitive. Then \begin{equation*} \left|\sum_{x\bmod{\ell^e}} \chi_1(F_1(x)) \cdots \chi_K(F_K(x))\right| \le (d-1) \ell^{e(1-1/d)},\end{equation*}
where $d = \sum_{k=1}^{K} \deg{F_k(T)}$.
\end{lem}

Let $\Delta(F)$ denote the discriminant of $F(T)$ if $F(0)=0$ and the discriminant of $T F(T)$ if $F(0)\ne 0$. Throughout this section and the next, we assume that $C(F)$ is fixed so large that primes exceeding $C(F)$ are odd and divide neither the leading coefficient of $F$ nor $\Delta(F)$. We also assume that $C(F) > (4D)^{2D+2}$ where $D=\deg{F(T)}$.

\begin{proof}[Verification of Hypothesis A] 

Suppose that $m\le x$ has $\gcd(f(m),q)=1$ and write $w$ for a value of $af(m)^{-1}$ mod $q$. 
Then $\#\V_{q, a, m}'' = \prod_{\ell^{e} \parallel q} V_{\ell^e}''$ and $\#\V_q' = \prod_{\ell^{e} \parallel q} V'_{\ell^e}$, where
\[ V''_{\ell^e} := \#\{(v_1,\dots,v_J)\bmod{\ell^e}: \gcd(v_1 \dots v_J,\ell)=1,~ \prod_{i=1}^{J} F(v_i) \equiv w\pmod{\ell^e}\} \] 
and
\[ V'_{\ell^e} := \#\{(v_1,\dots,v_J)\bmod{\ell^e}: \gcd(v_1 \dots v_J F(v_1)\cdots F(v_J),\ell)=1\}. \]
With $\chi_0$ denoting the principal Dirichlet character mod $\ell^e$,
\begin{align} \phi(\ell^e) V''_{\ell^e} &= \sum_{\chi \bmod{\ell^e}} \bar{\chi}(w) \sum_{v_1,\dots, v_J\bmod{\ell^e}}  \chi_0(v_1\cdots v_J)\chi(F(v_1) \cdots F(v_J)) \label{eq:orthogonalityapplication0}\\
&= V_{\ell^e}' + \sum_{\substack{\chi\bmod \ell^e \\ \chi \ne \chi_0}} \bar{\chi}(w) Z_{\chi}^{J}, \label{eq:orthogonalityapplication}\end{align}
where $Z_{\chi}:= \sum_{v \bmod{\ell^e}} \chi_0(v) \chi(F(v))$.
For each $\chi$ of conductor $\ell^{e_0}$ with $1\le e_0 \le e$, Lemma \ref{lem:combinedcharsum} gives $|Z_{\chi}| = \ell^{e-e_0} |\sum_{x\bmod{\ell^{e_0}}} \chi_0(x) \chi(F(x))| \le D \ell^{(e-e_0) + e_0 (1-1/(D+1))} = D \ell^{e-e_0/(D+1)}$. (If $\ell$ divides $F(0)$, then $\sum_{x\bmod{\ell^{e_0}}} \chi_0(x) \chi(F(x))= \sum_{x\bmod{\ell^{e_0}}}\chi(F(x))$, and we apply Lemma \ref{lem:combinedcharsum} with $k=1$ and $F_1(T)=F(T)$; 
otherwise we take $k=2$, $F_1(T)=T$, \and $F_2(T) = F(T)$.) As there are fewer than $\ell^{e_0}$ characters of conductor $\ell^{e_0}$, 
\[ \bigg|\sum_{\substack{\chi\bmod \ell^e \\ \chi \ne \chi_0}} \bar{\chi}(w) Z_{\chi}^{J}\bigg| \le \sum_{1\le e_0 \le e} \ell^{e_0} (D \ell^{e-e_0/(D+1)})^{J} = D^{J} \ell^{eJ} \sum_{1\le e_0 \le e} \ell^{e_0(1-J/(D+1))}.  \]
Since $J\ge D+2$ once $x$ is sufficiently large, each term in the sum $\sum_{1\le e_0 \le e} \ell^{e_0(1-J/(D+1))}$ 
is smaller than half the previous, and $\sum_{1\le e_0 \le e} \ell^{e_0(1-J/(D+1))} \le 2\ell^{1-J/(D+1)}$. Thus, $|\sum_{\substack{\chi\bmod \ell^e \\ \chi \ne \chi_0}} \bar{\chi}(w) Z_{\chi}^{J}| \le 2D^{J} \ell^{eJ} \ell^{1-J/(D+1)}$. Since $V_{\ell^e}' = (\phi(\ell^e)\alpha(\ell^e))^J$, we conclude from \eqref{eq:orthogonalityapplication} that
\begin{equation}\label{eq:generalVellrelation} \phi(\ell^e) V_{\ell^e}'' = V_{\ell^e}'(1 + R_\ell), \end{equation}
where
\[ |R_\ell| \le 2D^{J} \left(\frac{\ell^e}{\phi(\ell^e)} \alpha(\ell^e)^{-1}\right)^{J} \ell^{1-J/(D+1)} \le 2 (4D)^J \ell^{1-J/(D+1)} . \]
(We  use here that $\ell^e/\phi(\ell^e), \alpha(\ell^e)^{-1} \le 2$.) 
Multiplying over $\ell$ in \eqref{eq:generalVellrelation}, we see that Hypothesis A will follow if $(4D)^{J} \sum_{\ell \mid q} \ell^{1-J/(D+1)} = o(1)$. To check this last inequality, observe that when $x$ is large,
\begin{align*} (4D)^{J} \sum_{\ell \mid q} \ell^{1-J/(D+1)} &\le (4D)^{J} C(F)^{-J/(2D+2)} \sum_{\ell\mid q} \ell^{1-J/(2D+2)} \\ &\le (4D/C(F)^{1/(2D+2)})^{J} \sum_{\ell} \ell^{-2} < 2 (4D/C(F)^{1/(2D+2)})^{J};
\end{align*}
this last quantity tends to $0$ since $C(F) > (4D)^{2D+2}$ and $J\to\infty$.
\end{proof}

\begin{proof}[Verification of Hypothesis B] We follow the arguments for the corresponding step in \S\ref{sec:linear}.
Let $\xi(q)$ be the maximum number of roots $v$ mod $q$ of any congruence $F(v)\equiv a\pmod{q}$, where the maximum is over all residue classes $a\bmod{q}$. Then there are at most $\xi(q)$ possibilities for the residue class of $P$ modulo $q$ and our previous arguments yield
\begin{align*} N_{\inc}(q,a) &\le \xi(q) \frac{x}{\phi(q) (\log{x})^{1-\alpha\delta/2}} \exp(O((\log_3{x})^2 + (\log_2 q)^{O(1)})) \\ 
&< \xi(q) \frac{x}{\phi(q) (\log{x})^{1-2\alpha\delta/3}}. 
\end{align*}
This last quantity is certainly $o(N(q)/\phi(q))$ as long as $\xi(q) \ll (\log{x})^{(1-\delta)\alpha}$ (say). 
By the choice of $C(F)$, we have $\xi(q) \le D^{\omega(q)}$ for squarefree $q$, verifying Hypothesis B for squarefree $q$ having $\omega(q)\le (1-\delta)\alpha\log_2{x}/\log{D}$. 
On the other hand, by a result of Konyagin, each congruence $F(v)\equiv a\pmod{q}$ has $O(q^{1-1/D})$ roots modulo $q$ \cite{konyagin79b,konyagin79a}. Consequently, Hypothesis B also holds true for $q\le (\log{x})^{\alpha(1-\delta)(1-1/D)^{-1}}$, completing the proof of Theorem \ref{thm:generalF}.
\end{proof}

\section{Equidistribution along inputs with several prime factors exceeding $q$: Proof of Theorem \ref{thm:severallarge}}\label{sec:severallarge}

\begin{proof}[Proof of {\rm (a)}]

Recall that for the purposes of Theorem \ref{thm:severallarge}, we take $\delta:=1$ and $y = \exp((\log{x})^{1/2})$ in the framework developed in section \ref{sec:framework}. Lemma \ref{lem:nqvsnqcon} still applies to show that $N(q) \sim N_{\text{con}}(q)$ as $x \rightarrow \infty$, uniformly in $q \le (\log x)^K$ having $\alpha(q) \ne 0$. In particular, if $P_{D+2}(n) \le q$, then $P_{J}(n) < q \le y$ (once $x$ is large); thus $n$ is inconvenient, 
placing it in a set of size $o(N(q))$. It follows that the right-hand side of \eqref{eq:severallargeeq} is $\sim N(q)/\phi(q)$, and our task is that of showing the same for the left-hand side. The proof of Hypothesis A in \S\ref{sec:general} gives $N_{\con}(q,a) \sim N(q)/\phi(q)$. It remains only to show that there are $o(N(q)/\phi(q))$ inconvenient $n$ with $P_{D+2}(n) > q$ and $f(n)\equiv a\pmod{q}$.

As usual, we can assume $P(n) > z:=x^{1/\log_2 x}$ and that $n$ has no repeated prime factor exceeding $y=\exp(\sqrt{\log{x}})$. Since $n$ is inconvenient, we must have $P_J(n) \le y$. 
We suppose first that one of the largest $D+2$ primes in $n$ is repeated. Write $n = PSm$, where $P=P(n)$, $S$ is the largest squarefull divisor of $n/P$; hence, $Sm \le x/z$ and $S > q^2$. Given $S$ and $m$, there are fewer than $\pi(x/Sm) \ll x\log_2{x}/Sm \log x$ possibilities for $P$. Summing on squarefull $S > q^2$ bounds the number of $n$, given $m$, as  $\ll x\log_2{x}/qm \log x$. To handle the sum on $m$, write $m=AB$, where $A$ is the largest divisor of $m$ composed of primes exceeding $y$. Then $\Omega(A) < J$, while $B$ is $y$-smooth with $\gcd(f(B),q)=1$. Bounding $\sum 1/A$ and $\sum 1/B$ as in the proof of Lemma \ref{lem:nqvsnqcon}, we deduce that $\sum 1/m \le (\log x)^{\frac{1}{2}\alpha} \exp((\log_3 x)^{O(1)})$. 
Putting it all together, we see that the number of $n$ in this case is at most $\frac{x}{q (\log{x})^{1-\frac{1}{2}\alpha}} \exp((\log_3 x)^{O(1)})$, which is $o(N(q)/\phi(q))$.

We now suppose that each $P_i:=P_i(n)$ appears to the first power in $n$, for $i=1,2,\dots,D+2$, and we write $n = P_1 \cdots P_{D+2} m$. Since $f(n)\equiv a\pmod{q}$, it must be that $\gcd(f(m),q)=1$. Furthermore, letting $w$ denote a value of $af(m)^{-1}$ mod $q$, 
\[ (P_1,\dots, P_{D+2}) \bmod q \in \Vw, \] 
where
\[ \Vw := \{(v_1,\dots,v_{D+2})\bmod{q}: \gcd(v_1\cdots v_{D+2},q)=1,~ F(v_1)\cdots F(v_{D+2}) \equiv w \pmod{q}\}. \]

Let us estimate the size of $\#\Vw$. Put
\[ V_{\ell^e} = \#\{(v_1,\dots, v_{D+2})\bmod{\ell^e}: \gcd(v_1 \cdots v_{D+2},\ell)=1,~F(v_1)\cdots F(v_{D+2}) \equiv w \pmod{\ell^e}\},
\]
so that $\#\Vw = \prod_{\ell^e \parallel q} V_{\ell^e}$.
From the proof of \eqref{eq:generalVellrelation}, with $J$ replaced by $D+2$,
\[ \phi(\ell^e) V_{\ell^e} = (\alpha(\ell^e) \phi(\ell^e))^{D+2} (1+R_{\ell}), \]
where $|R_{\ell}| \le 2(4D)^{D+2} \ell^{-1/(D+1)} \ll \ell^{-1/(D+1)}$.
Multiplying on $\ell$ gives
\begin{align} \phi(q) \#\Vw &\ll \alpha(q)^{D+2} \phi(q)^{D+2} \exp\bigg(O\big(\sum_{\ell\mid q} \ell^{-1/(D+1)}\big)\bigg)\notag \\
&\ll \phi(q)^{D+2} \exp(O((\log{q})^{1-1/(D+1)})).\label{eq:sizevestimate} \end{align}

Given $P_2,\dots, P_{D+2}$, $m$, and $\v = (v_1,\dots,v_{D+2})\bmod{q} \in \Vw$, the number of possibilities for $P_1$ is $\ll x\log_2 x/\phi(q)mP_2\cdots P_{D+2}\log{x}$, 
by Brun--Titchmarsh. Summing on $P_2,\dots,P_{D+2}$, we see that the number of possibilities for $n$ given $\v$ and $m$ is $\ll x(\log_2{x})^{O(1)}/\phi(q)^{D+2}m\log{x}$. (We use here
that
\[ \sum_{\substack{q < p \le x \\ p\equiv v\pmod{q}}}\frac{1}{p} \ll \frac{\log_2{x}}{\phi(q)}, \]
uniformly in the choice of $v$, which follows from Brun--Titchmarsh and partial summation; alternatively, one can apply Lemma \ref{lem:NP}.) We sum on $\v \in \Vw$, using \eqref{eq:sizevestimate}, and then sum on $m$, writing $m=AB$ and making the estimates as earlier in this proof. We find that the total number of $n$ is at most 
\[ \frac{x}{ \phi(q) (\log{x})^{1-\frac{1}{2}\alpha}} \exp(O((\log_2{x})^{1-1/(D+1)})), \]
which is $o(N(q)/\phi(q))$.\end{proof}

\begin{proof}[Proof of \rm{(b)}] We follow the proof of (a), replacing $D+2$  everywhere by $2$. It suffices to show that
\begin{equation}\label{eq:2wanted} \phi(\ell) V_{\ell} \le \phi(\ell)^2 (1+O(1/\sqrt{\ell})) \end{equation}
for each $\ell$, for then $\phi(q)\#\Vw \ll \phi(q)^2 \exp(O((\log{q})^{1/2}))$, which is a suitable analogue of \eqref{eq:sizevestimate}.

Certainly $V_{\ell}$ is bounded by the count of $\F_{\ell}$-points on the affine curve $F(x)F(y)=w$. 

The polynomial $F(x)F(y)-w$ is absolutely irreducible over $\F_{\ell}$.\footnote{The published version of the paper contained an incorrect argument for this claim.} Indeed, suppose that $F(x)F(y)-w = U(x,y) V(x,y)$ for some $U(x,y), V(x,y) \in \overline{\F}_{\ell}[x,y]$. Then for each root $\theta \in \overline{\F}_{\ell}$ of $F$, we find that $-w = U(\theta,y) V(\theta,y)$, and so in particular $U(\theta,y)$ is constant. Thus, if we write
\[ U(x,y) = \sum_{k\ge 0} a_k(x)y^k, \] with each $a_k(x) \in \overline{\F}_{\ell}[x]$, then $a_k(\theta)=0$ for each $k>0$. Since $F$ has no multiple roots over $\overline{\F}_{\ell}$, each such $a_k(x)$ is forced to be a multiple of $F(x)$, hence $U(x, y) \equiv a_0(x)\pmod{F(x)}$. A symmetric argument shows that $V(x, y) \equiv b_0(y) \pmod{F(y)}$ for some $b_0(y) \in \overline{\F}_{\ell}[y]$, so that $V(x, \theta) = b_0(\theta)$. Consequently, for any root $\theta \in \overline{\F}_{\ell}$ of $F$, 
\[-w \equiv F(x)F(\theta) - w \equiv U(x, \theta) V(x, \theta) \equiv a_0(x)b_0(\theta) \pmod{F(x)},\]
which shows that $U(x, y) \equiv a_0(x) \equiv c \pmod{F(x)}$ for some constant $c \in \overline{\F}_{\ell}$. But this forces $c = U(\theta, \theta)$, showing that $F(x)$ divides $U(x, y) - U(\theta, \theta)$. By symmetry, so does $F(y)$, and we obtain $U(x, y) = U(\theta, \theta) + F(x)F(y) Q(x, y)$ for some $Q(x,y) \in \overline{\F}_{\ell}[x,y]$.
Degree considerations now imply that for $U(x,y)$ to divide $F(x)F(y)-w$, either $Q(x,y)$ is a nonzero constant, in which case $V(x,y)$ is constant, or $Q(x,y)=0$, in which case $U(x,y)$ is constant.  

Now we apply the version of the Hasse--Weil bound appearing as \cite[Corollary 2(b)]{LY94}; this gives that the number of $\F_{\ell}$-points on $F(x)F(y)=w$ 
is at most $\ell + 1 +\frac{1}{2}(2D-1)(2D-2)\lfloor 2\sqrt{\ell}\rfloor$,
which is $\phi(\ell) (1+O(1/\sqrt{\ell}))$, yielding \eqref{eq:2wanted}. 
\end{proof}

\section{Concluding remarks and further questions}

Elementary methods often enjoy a robustness surpassing their analytic counterparts, and our (quasi)elementary approach to weak uniform distribution is no exception. Not only does our method yield a range of uniformity in $q$ wider than that (seemingly) accessible to more `obvious' attacks via mean value theorems for multiplicative functions, but the method applies to functions that do not fit conveniently into the `multiplicative managerie'. We illustrate with the following theorem; note that the distribution in residue classes of the function $A^\ast(n)$ below does not seem easily approached via mean value theorems.

\begin{thm}\label{thm:PDSumAndAlt}
Fix $K \geq 1$. The sum of prime divisors function $A(n) := \sum_{j=1}^{\Omega(n)} P_j(n)$, as well as the alternating sum of prime divisors function $A^*(n) := \sum_{j=1}^{\Omega(n)} (-1)^{j-1} P_j(n)$, is asymptotically uniformly distributed to all moduli $q \le (\log x)^K$. In other words, as $x \rightarrow \infty$, 
\begin{equation}\label{eq:aastarUD}
\sum_{\substack{n \leq x\\A(n) \equiv a \pmod q}} 1 ~\sim~ \sum_{\substack{n \leq x\\A^*(n) \equiv a \pmod q}} 1  \hspace{3mm} ~\sim~ \frac xq,
\end{equation}
uniformly in moduli $q \le (\log x)^K$ and residue classes $a \bmod q$.
\end{thm}

\begin{rmk} The uniform distribution of $A(n)$ mod $q$ for each fixed $q$ is a consequence of the theorem of Delange quoted in the introduction, with more precise results appearing in work of Goldfeld \cite{goldfeld17}. For varying $q$, the problem seems to have been first considered in \cite{PSR}; there Hal\'asz's mean value theorem is used to show uniform distribution of $A(n)$ mod $q$ for $q\le (\log{x})^{\frac12-\delta}$ (for any fixed $\delta>0$), a significantly narrower range than that allowed by Theorem \ref{thm:PDSumAndAlt}.\end{rmk}

\begin{proof}[Proof of Theorem \ref{thm:PDSumAndAlt}] 

With $y:= \exp(\sqrt{\log x})$, arguments analogous to (but simpler than) those in the proof of Lemma \ref{lem:nqvsnqcon} show that the number of inconvenient $n\le x$ is $o(x)$, while  arguments analogous to (but simpler than) those in the verification of Hypothesis B of \S\ref{sec:linear} show that the number of inconvenient $n \le x$ having $A(n) \equiv a \pmod q$ or $A^*(n) \equiv a \pmod q$ is $o(x/q)$. Hence, it suffices to show that
\begin{equation}\label{eq:convenientAAstar}
N(q, a) \sim N^*(q, a) \hspace{3mm} ~\sim~ \frac 1q \sum_{\text{convenient } n\le x} 1,
\end{equation}
where $N(q, a)$ (respectively, $N^*(q, a)$) denotes the number of convenient $n \le x$ having $A(n) \equiv a \pmod q$ (resp., $A^*(n) \equiv a \pmod q$). 

Proceeding as in  \S\ref{sec:framework}, we define, for an arbitrary residue class $w$ mod $q$,
\[\Vw := \{(v_1, \dots, v_J)\bmod{q}: \gcd(v_1 \dots v_J,q)=1,~\sum_{j=1}^J v_j \equiv w \pmod{q}\}\]
and 
\[\Vwst := \{(v_1, \dots, v_J)\bmod{q}: \gcd(v_1 \dots v_J,q)=1,~\sum_{j=1}^J (-1)^{j-1} v_j \equiv w\pmod{q}\}, \]
and we write
\begin{equation*}
N(q, a) = \sum_{m \le x} \frac{1}{J!} \sum_{\mathbf{v} \in \Vm}  \sum_{\substack{P_1,\dots, P_J\text{ distinct} \\ P_1\cdots P_J \le x/m \\ \text{each } P_j > L_m \\ \text{each } P_j \equiv v_j\pmod{q}}} 1, \qquad N^*(q, a) = \sum_{m \le x} \frac{1}{J!} \sum_{\mathbf{v} \in \VmSt}  \sum_{\substack{P_1,\dots, P_J\text{ distinct} \\ P_1\cdots P_J \le x/m \\ \text{each } P_j > L_m \\ \text{each } P_j \equiv v_j\pmod{q}}} 1,  
\end{equation*}
where $\Vm:= \V_q(a-A(m))$ and $\VmSt := \V_q^*(a - (-1)^J A^*(m))$.

By $J$ applications of Siegel-Walfisz, we now obtain  
\begin{align}
N(q, a) := \sum_{m\le x}\frac{\#\Vm}{\phi(q)^{J}}\Bigg(\frac{1}{J!} \sum_{\substack{P_1,\dots, P_J\text{ distinct} \\ P_1\cdots P_J \le x/m \\ \text{each } P_j > L_m}} 1\Bigg) + O\left(x \exp\left(-\frac{1}{8} C_K (\log{x})^{1/4}\right)
\right) \label{NSWExpr}\\
N^*(q, a) := \sum_{m\le x}\frac{\#\VmSt}{\phi(q)^{J}}\Bigg(\frac{1}{J!} \sum_{\substack{P_1,\dots, P_J\text{ distinct} \\ P_1\cdots P_J \le x/m \\ \text{each } P_j > L_m}} 1\Bigg) + O\left(x \exp\left(-\frac{1}{8} C_K (\log{x})^{1/4}\right)
\right), \label{N^*SWExpr}
\end{align}
for some constant $C_K>0$ depending only on $K$. As an analogue of our Hypothesis A, we claim that as $x \rightarrow \infty$, 
\begin{flalign}
\#\Vm \sim (\1_{2 \nmid q} + 2\cdot \1_{2\mid q,\,J \equiv a-A(m) \pmod 2}) \frac{\phi(q)^J}q, \label{VmNo}\\
\#\VmSt \sim (\1_{2 \nmid q} + 2\cdot \1_{2\mid q,\,J \equiv a-(-1)^J A^*(m) \pmod 2}) \frac{\phi(q)^J}q, \label{VmstNo}
\end{flalign} 
uniformly in $m \le x$ and in $q \le (\log x)^K$. (If
$\1_{2 \nmid q} + 2\cdot \1_{2\mid q,\,J \equiv a-A(m) \pmod 2}=0$, the asymptotic \eqref{VmNo} should be interpreted as the claim $\V_{q, a, m}$ is empty, and similarly for \eqref{VmstNo}.) To this end, it suffices to show that
\begin{equation}\label{VwVw*No}
\#\Vwst = \#\Vw \sim (\1_{2 \nmid q} + 2\cdot \1_{2\mid q,\,J \equiv w \pmod 2}) \frac{\phi(q)^J}q,
\end{equation}
uniformly in $q \le (\log x)^K$ and in residue classes $w$ mod $q$.
The equality in \eqref{VwVw*No} follows immediately from the one-to-one correspondence $(v_1, \cdots, v_J) \xleftrightarrow{} (v_1, -v_2, \cdots, (-1)^{J-1} v_J)$ between $\Vw$ and $\Vwst$.  To see the asymptotic, we write $\#\Vw = \prod_{\ell^e \parallel q} V_{\ell^e}$, where for each prime power $\ell^e \parallel q$, 
\begin{align*}
V_{\ell^e} :&= \#\{(v_1, \dots, v_J)\bmod{\ell^e}: \gcd(v_1 \dots v_J,\ell)=1,~\sum_{j=1}^J v_j \equiv w \pmod{\ell^e}\}\\
&=\frac{\phi(\ell^e)^J}{\ell^e} + \frac1{\ell^e} \sum_{0<r<\ell^e} \exp\left(-\frac{2 \pi i rw}{\ell^e}\right) S_\ell(r)^J, 
\end{align*}
with $S_\ell(r) := \sum_{\substack{v \bmod \ell^e,~(v, \ell)=1}} \exp(2 \pi i rv/\ell^e)$  (a Ramanujan sum). Since $S_\ell(r) = \1_{\ell^{e-1} \parallel r} (-\ell^{e-1})$ for all $r \in \{1, \cdots, \ell^e-1\}$ (see, for instance, \cite[Theorem 4.1, p.\ 110]{MV07}), we deduce that as $x \rightarrow \infty$,
\[\#\Vw = (\1_{2 \nmid q} + 2\cdot\1_{2\mid q,\,J \equiv w \pmod 2}) \frac{\phi(q)^J}q  \prod_{\substack{\ell|q\\ \ell>2}} \left(1 + O\left(\frac1{(\ell-1)^{J-1}}\right)\right),\]
leading to \eqref{VwVw*No}, since $\sum_{\ell\mid q,\,\ell>2} 1/(\ell-1)^{J-1} = o(1)$ as $J \rightarrow \infty$. 

Plugging \eqref{VmNo} and \eqref{VmstNo} into \eqref{NSWExpr} and \eqref{N^*SWExpr} respectively, and carrying out our initial reductions in reverse order completes the proof of \eqref{eq:convenientAAstar}, and hence also that of \eqref{eq:aastarUD}, for odd $q \le (\log x)^K$. On the other hand, when $q$ is even we obtain
\[ N(q, a) = \frac2q\sum_{\substack{n \le x\\A(n) \equiv a \pmod 2}} 1 +  o\left(\frac xq\right), \quad N^*(q, a) = \frac2q\sum_{\substack{n \le x\\A^*(n) \equiv a \pmod 2}} 1 + o\left(\frac xq\right);\]
here, it has been noted that $a-A(m) \equiv J \pmod 2$ is equivalent to $A(mP_1 \cdots P_J) \equiv a \pmod 2$, and likewise for $A^*$ in place of $A$. Since $A(n)$ is known to be equidistributed mod $2$ (as discussed in the remarks preceding the theorem), and $A^*(n) \equiv A(n) \pmod 2$, the theorem follows.
\end{proof}

The flexibility of our method suggests the possibility of extensions in several different directions. One natural generalization is to study simultaneous weak equidistribution for a finite family of polynomially-defined multiplicative functions. Problems of this kind with fixed moduli were investigated by Narkiewicz in \cite{narkiewicz82}, and initial results towards uniformity were obtained in \cite{PSR22}. It should now be possible to draw more complete conclusions. Going in a different direction, one could apply our method to additive functions, aiming perhaps at a uniform generalization of the quoted theorem of Delange. One could even consider simultaneous equidistribution of additive and multiplicative functions; here estimates for hybrid character sums, as in \cite{cochrane02}, should prove useful. 

We close on a more speculative note. The mixing exploited in this paper can be interpreted as a quantitative ergodicity phenomenon for random walks on multiplicative groups. However, our proofs go through character sum estimates; one might say that no actual Markov chains were harmed in the production of our arguments. It would be interesting to investigate the extent to which the (rather substantially developed) theory of Markov chain mixing could be brought directly to bear on these kinds of uniform and weak uniform distribution questions. This has the potential to open up applications in situations where character sum technology is unavailable.

\section*{Acknowledgements} We thank the referee for carefully reading the manuscript and for making helpful suggestions that have improved the results and the exposition. The first named author (P.P.) is supported by NSF award DMS-2001581. 

\providecommand{\bysame}{\leavevmode\hbox to3em{\hrulefill}\thinspace}
\providecommand{\MR}{\relax\ifhmode\unskip\space\fi MR }
\providecommand{\MRhref}[2]{%
  \href{http://www.ams.org/mathscinet-getitem?mr=#1}{#2}
}
\providecommand{\href}[2]{#2}

\end{document}